\documentclass[11pt, twoside, a4paper]{article}
\usepackage[english]{babel}
\usepackage{amsmath,amssymb,amsthm,epsfig,color}
\begin{document}

\newtheorem{theorem}{Theorem}
\newtheorem{proposition}[theorem]{Proposition}
\newtheorem{lemma}[theorem]{Lemma}
\newtheorem{definition}{Definition}

\newtheorem*{theoremA}{A Banach--Caccioppoli-type Theorem}

\newcommand{\keywords}[1]{\par\addvspace\baselineskip\noindent\textbf{Keywords:}\enspace\ignorespaces#1}

\newcommand{\AMSclassification}[1]{\par\addvspace\baselineskip\noindent\textbf{MSC:}\enspace\ignorespaces#1}

\title{The effective potential and transshipment in thermodynamic formalism at temperature zero}
\author{
Eduardo Garibaldi\thanks{Supported by PROCAD UNICAMP-UFRGS 162-2007.}\\
\small{Departamento de Matem\'atica}\\
\small{Universidade Estadual de Campinas}\\
\small{13083-859 Campinas - SP, Brasil}\\
\small{\texttt{garibaldi@ime.unicamp.br}}
\and
Artur O. Lopes\thanks{Partially supported by CNPq, PRONEX -- Sistemas Din\^amicos, INCT em Matem\'atica, and beneficiary of CAPES financial support (PROCAD UFRGS-IMPA and PROCAD UNICAMP-UFRGS).}\\
\small{Instituto de Matem\'atica}\\
\small{Universidade Federal do Rio Grande do Sul}\\
\small{91509-900 Porto Alegre - RS, Brasil}\\
\small{\texttt{arturoscar.lopes@gmail.com}} }
\date{\today}
\maketitle

\vspace{-.2cm}

\begin{abstract}
Denote the points in $\{1,2,..,r\}^\mathbb{Z}= \{1,2,..,r\}^\mathbb{N}\times
\{1,2,..,r\}^\mathbb{N}$ by $(\mathbf{y}^*, \mathbf{x})$. Given a Lipschitz
continuous observable $A: \{1,2,..,r\}^\mathbb{Z} \to \mathbb{R} $, we define
the map $\mathcal{G}^+: \mathcal{H }\to \mathcal{H }$ by
\begin{equation*}
\mathcal{G}^+ (\phi)(\mathbf{y}^*) = \sup_{\mu \in \mathcal{M}_\sigma} \left[
\int_{\{1,2,..,r\}^\mathbb{N}} \left( A(\mathbf{y}^*, \mathbf{x}) + \phi(%
\mathbf{x})\right) \; d\mu(\mathbf{x}) + h_\mu(\sigma) \right],
\end{equation*}
where:

\noindent -- $\sigma$ is the left shift map acting on $\{1,2,..,r\}^\mathbb{N%
}$;

\noindent -- $\mathcal{M}_\sigma$ denotes the set of $\sigma$-invariant
Borel probabilities;

\noindent -- $h_\mu(\sigma) $ indicates the Kolmogorov-Sinai entropy;

\noindent -- $\mathcal{H }$ is the Banach space of Lipschitz real-valued
functions on $\{1,2,..,r\}^\mathbb{N} $.

We show there exist a unique $\phi^+ \in \mathcal{H }$ and a unique $%
\lambda^+\in \mathbb{R}$ such that
\begin{equation*}
\mathcal{G}^+ (\phi^+) = \phi^+ + \lambda^+.
\end{equation*}
We say that $\phi^+ $ is the effective potential associated to $A$. This
also defines a family of $\sigma$-invariant Borel probabilities $\mu_{%
\mathbf{y}^*}$ on $\{1,2,..,r\}^\mathbb{N}$, indexed by the points $\mathbf{y%
}^* \in \{1,2,..,r\}^\mathbb{N}.$
Finally, for $A$ fixed and for variable positive real values $\beta$, we
consider the same problem for the Lipschitz observable $\beta A$. We
investigate then the asymptotic limit when $\beta\to \infty $ of the
effective potential (which depends now on $\beta$) as well as the above
family of probabilities. We relate the limit objects with
an ergodic version of Kantorovich transshipment problem. In statistical
mechanics $\beta \propto 1/T$, where $T$ is the absolute temperature. In
this way, we are also analyzing the problem related to the effective
potential at temperature zero.

\vspace{-.2cm}


\vspace{-.2cm}

\AMSclassification{37A05, 37A35, 37A60, 37D35, 49K27, 49Q20, 82B05, 90C05, 90C50.}
\end{abstract}

\newpage

\begin{section}{Introduction}

Our purpose is to propose, in a rigorous mathematical way, a description of
the main features of what could be called in statistical mechanics the
\emph{effective potential formalism} for long range interactions. In this way, we are able to present
a family of \emph{effective probabilities}, each one corresponding to a Gibbs state in the sense
of Ruelle's thermodynamic setting \cite{PP}. We also consider the
limit behavior of this family of probabilities when the
temperature goes to zero. In this case, we relate our analysis
with a kind of ergodic  Kantorovich transshipment problem.
We point out that in the classical transport theory \cite{Ra, Vi1, Vi2}
there is no assumptions about invariant probabilities playing a role in the problem.

Actually our framework will be more general than Bernoulli shifts.
We will develop the theory using one-sided topologically transitive
subshifts of finite type defined by symmetric transition matrices.

Hence let $ \mathbf M : \{1,\ldots, r\} \times \{1, \ldots, r\} \to \{0, 1 \} $ be an irreducible transition matrix.
One has naturally two subshifts associated to such a matrix. We can introduce the standard subshift of finite type
$$ \Sigma_{\mathbf M} = \left \{ (x_0, x_1, \ldots) \in \{1, \ldots, r\}^{\mathbb Z_+} : \mathbf M(x_j, x_{j + 1}) = 1 \right \}, $$
as well as the dual subshift of finite type
$$ \Sigma^*_{\mathbf M^T} = \left \{ (\ldots, x_1, x_0) \in \{1, \ldots, r\}^{\mathbb Z_-} : \mathbf M^T(x_j, x_{j + 1}) = 1 \right \}. $$
As topological spaces, both subshifts are always compact metrizable spaces.
We suppose henceforth that the matrix $ \mathbf M $ is symmetric. So we have a canonical homeomorphism
$ \mathbf x = (x_0, x_1, \ldots) \in \Sigma_{\mathbf M} \mapsto \mathbf x^* = (\ldots, x_1, x_0) \in \Sigma^*_{\mathbf M} $.

Given $ \Lambda \in (0, 1) $, we equip as usual $ \Sigma_{\mathbf M} $ with the metric $ d(\mathbf x, \mathbf y) = \Lambda^k $, where
$ \mathbf x = (x_0, x_1, \ldots), \mathbf y = (y_0, y_1, \ldots) \in \Sigma_{\mathbf M} $ and $ k = \min \{j: x_j \ne y_j \} $.
Hence, for $ \mathbf x^*, \mathbf y^* \in \Sigma^*_{\mathbf M} $, we just set $ d^*(\mathbf x^*, \mathbf y^*) := d(\mathbf x, \mathbf y) $.

Let $ \sigma $ be the left shift map acting on $ \Sigma_{\mathbf M} $ and let $ \sigma^* $ be the right shift map acting on
$ \Sigma^*_{\mathbf M} $, namely,
$$ \sigma(x_0, x_1, x_2, \ldots) = (x_1, x_2, \ldots) \; \text{ and } \; \sigma^*(\ldots, x_2, x_1, x_0) = (\ldots, x_2, x_1). $$
Clearly, $ * \circ \sigma = \sigma^* \circ * $. Furthermore, since $ \mathbf M $ is irreducible, notice that
the dynamics $ (\Sigma_{\mathbf M}, \sigma) $ is transitive -- and consequently the conjugated dynamical system $ (\Sigma^*_{\mathbf M}, \sigma^*) $ too.

Let $ C^0(\Sigma_{\mathbf M}) $ and $ C^0(\Sigma^*_{\mathbf M}) $ denote the spaces of continuous real-valued functions
on respectively $ \Sigma_{\mathbf M} $ and $ \Sigma^*_{\mathbf M} $, both equipped with the topology of uniform convergence.
Thus, we can obtain from the previous homeomorphism an isometry $ * : C^0(\Sigma_{\mathbf M}) \to C^0(\Sigma^*_{\mathbf M}) $
writing $ f^*(\mathbf x^*) := f(\mathbf x) $ for every function $ f \in C^0(\Sigma_{\mathbf M}) $. This fact allows us to make the
identification $ C^0(\Sigma_{\mathbf M}) \simeq C^0(\Sigma^*_{\mathbf M}) $.

The same isometric property is verified for either H\"older or Lipschitz continuous real-valued functions.
Since one can simply incorporate the H\"older exponent into the distance, we remark that to work with the Lipschitz class does not lead to
loss of generality. Therefore, $ \mathcal H $ will denote in this article the Banach space of Lipschitz continuous real-valued functions
on either $ \Sigma_{\mathbf M} $ or $ \Sigma^*_{\mathbf M} $, equipped with the norm $ \| \cdot \|_{\mathcal H} := \| \cdot \|_0 + \text{Lip}(\cdot) $,
where $ \| \cdot \|_0 $ denotes the uniform norm and
$$ \text{Lip}(\phi) = \sup_{d(\mathbf x, \mathbf y) > 0} \frac{\left| \phi(\mathbf x) - \phi(\mathbf y) \right|}{d(\mathbf x, \mathbf y)} =
\sup_{d^*(\mathbf x^*, \mathbf y^*) > 0} \frac{\left| \phi^*(\mathbf x^*) - \phi^*(\mathbf y^*) \right|}{d^*(\mathbf x^*, \mathbf y^*)} =
\text{Lip}(\phi^*). $$

Using the standard subshift $ \Sigma_{\mathbf M} $ and its dual $ \Sigma^*_{\mathbf M} $, one may easily introduce its natural
invertible extension $ (\hat \Sigma_{\mathbf M}, \hat \sigma) $:
$$ \hat \Sigma_{\mathbf M} = \big \{ (\mathbf y^*, \mathbf x) \in \Sigma^*_{\mathbf M} \times \Sigma_{\mathbf M} : \mathbf M (y_0, x_0) = 1 \big \}, $$
$$ \hat \sigma(\ldots, y_1, y_0 | x_0, x_1, \ldots) = (\ldots, y_0, x_0 | x_1, x_2, \ldots). $$

Denote by $ \mathcal M_\sigma $ the weak* compact and convex set of $\sigma$-invariant Borel probability measures.
For any $ \mu \in \mathcal M_\sigma $, let $ h_\mu(\sigma) $ indicate the Kolmogorov-Sinai entropy.

\begin{definition}
Given a Lipschitz continuous function $ A : \hat \Sigma_{\mathbf M} \to
\mathbb R $, we consider then the map $ \mathcal G^+=\mathcal
G^+_A : \mathcal H \to \mathcal H $ defined\footnote{Notice that a more rigorous definition would consider
$ \int_{\Sigma_{\mathbf M}} \left( A(\mathbf y^*, \mathbf x) \mathbf M(\mathbf y^*, \mathbf x) + \phi(\mathbf x)\right) \; d\mu(\mathbf x) $,
where $ \mathbf M(\mathbf y^*, \mathbf x) := \mathbf M(y_0, x_0) $ for any point $ (\mathbf y^*, \mathbf x) = (\ldots, y_1, y_0 | x_0, x_1, \ldots) $.
We prefer to simplify the notation.}\label{multiplicacao M} by
$$ \mathcal G^+ (\phi)(\mathbf y^*) = \sup_{\mu \in \mathcal M_\sigma}
\left[ \int_{\Sigma_{\mathbf M}} \left( A(\mathbf y^*, \mathbf x) + \phi(\mathbf x)\right) \; d\mu(\mathbf x) + h_\mu(\sigma) \right] $$
\end{definition}

It is not difficult to see that $ \text{Lip}(\mathcal G^+ (\phi)) \le \| A \|_0 + \text{Lip}(A) $ for all $ \phi \in \mathcal H $.
Furthermore, thanks to the characterization via variational principle of the topological pressure $ P_{_{TOP}} : \mathcal H \to \mathbb R $, that is,
$$ P_{_{TOP}}(\phi) = \max_{\mu \in \mathcal M_\sigma} \left[ \int_{\Sigma_{\mathbf M}} \phi(\mathbf x) \; d\mu(\mathbf x) + h_\mu(\sigma) \right]
\quad \forall \; \phi \in \mathcal H, $$
we immediately get
$$ \mathcal G^+(\phi)(\mathbf y^*) = P_{_{TOP}}(A(\mathbf y^*, \cdot) + \phi). $$
In particular, thanks to the Ruelle-Perron-Frobenius Theorem,
for each $ \mathbf y^* \in \Sigma^*_{\mathbf M} $, there exists a unique
probability $ \mu_{\mathbf y^*} \in \mathcal M_\sigma $ (the equilibrium state associated to $ A(\mathbf y^*, \cdot) + \phi \in \mathcal H $)
achieving the supremum in the definition of the value $ \mathcal G^+(\phi)(\mathbf y^*) $.

As a physical motivation to analyze the above problem, we mention the paper by W. Chou and R. Griffiths \cite{CG}.
They study ground states of one-dimensional systems, in particular of a very common model in solid state physics: the Frenkel-Kontorova model (specific applications are presented in section VI).
They consider a certain model which depends on temperature and which has a natural potential. But due to interaction and temperature, there exists another potential, called the {\it effective potential}, which plays the essential role in the problem. The limit when temperature goes to zero is considered in section III B. Their expression (3.16) may be seen as a min-plus version of the one in our main theorem. See also \cite{Ba,CD,Co} for more details on additive eigenvalue problems.

The entropy penalization method was considered in \cite{GV} and \cite{GLM} (see the main properties on these references) in the setting of Aubry-Mather theory.
In \cite{LMST}, questions also related to the article of Chou and Griffiths were analyzed in the context of Markov chains on the interval.
The problem we consider here has similarities. Nevertheless, we point out that our entropy is Kolmogorov-Sinai entropy, which has a dynamical character. As we said before, our setting is the one of thermodynamic formalism \cite{PP}. Finally, the relation of the effective action problem with the ergodic Kantorovich transshipment problem
(see section 3), as far as we know, is completely new.

As another physical motivation for the study of the above problem, we mention section 2.5 of Salmhofer's book \cite{Sal}.
The function $ \phi $ plays there the role of a chemical potential. In \cite{Sal}, using another
notation, the expression (2.103) of the effective action for the interaction $-\lambda V$ and the propagator $C$
$$ \int e^{- \lambda V(\phi) + (C^{-1} \psi,\phi)} \, d \mu_C(\phi), \quad \text{for a fixed $\psi$,} $$
should be read, under our notation, as
$$ \int e^{\phi(\mathbf x) + A(\mathbf y^*, \mathbf x)} \, d \mu_C(\mathbf x), \quad \text{for a fixed $\mathbf y^*$.} $$
Note that the above probability maximizes pressure when one
considers, for each fixed $\mathbf y^* $, the observable $ A(\mathbf y^*, \cdot) + \phi $,
and the corresponding  variational problem where the entropy
$h(\nu|\mu_C)$ of a given $\nu$ is consider relative to a fixed
initial probability $\mu_C$. As we are in the framework of
thermodynamic formalism, we do not consider relative entropy, but
Kolmogorov-Sinai entropy.

Our main result can be stated as follows.

\begin{theorem}\label{principal}
Suppose $A: \hat \Sigma_{\mathbf M} \to \mathbb R$ is a Lipschitz continuous observable.
Then there exist a unique function $ \phi^+ \in \mathcal H $ (up to an additive constant)
and a unique constant $ \lambda^+ \in \mathbb R $ such that
$$ \mathcal G^+ (\phi^+) = \phi^+ + \lambda^+. $$
\end{theorem}

We point out that \cite{GLM, GV, LMST} consider a
similar problem but for the so called entropy penalization method.
The proof of this theorem will be presented in the end of the paper.
Obviously the function $\phi^+$ and the constant $ \lambda^+ $ in the previous statement depend on $A$.

\begin{definition}
Given a Lipschitz continuous observable $ A: \hat \Sigma_{\mathbf M} \to \mathbb R $, we say that a constant $\lambda^{+} \in \mathbb{R}$
is the effective constant for $ A $ if there exists a function $ \phi^{+} \in \mathcal H $ such that
$$ \mathcal G^+ (\phi^+) = \phi^+ + \lambda^+. $$
Any such a function $ \phi^{+} $ is called a (forward) effective potential for $ A $.
\end{definition}

\begin{definition}
Given a Lipschitz continuous observable $ A: \hat \Sigma_{\mathbf M} \to \mathbb R $ and a point $ \mathbf y^* \in \Sigma_{\mathbf M}^*$,
we say that the unique $\sigma$-invariant probability $ \mu_{\mathbf y^*} = \mu_{\mathbf y^*, A}$ on $ \Sigma_{\mathbf M} $ such that
$$ \int_{\Sigma_{\mathbf M}} \left( A(\mathbf y^*, \mathbf x) + \phi^+ (\mathbf x)\right) \; d\mu_{\mathbf y^*}(\mathbf x) + h_{\mu_{\mathbf y^*}}(\sigma) =
\phi^{+}(\mathbf y^*) + \lambda^{+} $$
is the effective probability for $ A $ at $ \mathbf y^*$, where $\phi^{+}$
and $\lambda^{+}$ are the effective ones associated to $A$. In this way, we get a family of Gibbs states on the variable $ \mathbf x$
indexed by $\mathbf y^*$.
\end{definition}

For a fixed $A$ as above, we consider a positive  parameter $\beta$, the observable $\beta A$, and
the corresponding $ \phi^{+}_\beta$, $\lambda^{+}_\beta$ and
$ \{ \mu_{\mathbf y^*, \beta A} \}_{\mathbf y^* \in \Sigma_{\mathbf M}^*} $.
 We investigate then the limit problem when $\beta \to \infty$, showing the existence
(in the uniform topology) of accumulation Lipschitz functions for the family $ \{\phi^{+}_\beta / \beta\}_{\beta>0} $,
characterizing the accumulation probabilities of $ \{ \mu_{\mathbf y^*, \beta A} \}_{\beta>0} $ for each $ \mathbf y^* $,
and proving that $\lambda^{+}_\beta / \beta $ converges (see section 2).

We remark at last that one could also consider the (backward) transformation
$ \mathcal G^{-}=\mathcal G^{-}_A : \mathcal H \to \mathcal H $ defined by
$$ \mathcal G^{-} (\phi)(\mathbf x) = \sup_{\mu \in \mathcal M_{\sigma^*}}
\left[ \int_{\Sigma_{\mathbf M}^*} \left( A(\mathbf y^*, \mathbf x) + \phi(\mathbf y^*)\right) \; d\mu(\mathbf y^*) + h_\mu(\sigma^*) \right] ,$$
and all analogous results could be easily stated and similarly proved.

The structure of the paper is the following: in section 2 we discuss the thermodynamic properties of the effective objects,
 in section 3 we consider the ergodic Kantorovich transshipment problem (which appears in a natural way when the
temperature goes to zero), and finally in section 4 we present the proof of the main theorem.

\end{section}

\begin{section}{Thermodynamic formalism at temperature zero}

The analysis of the thermodynamic formalism for a given observable
$A$ at temperature zero is, by definition, the study of the limit
of Gibbs probabilities associated to $A$ at temperature $T$,
that is, for $\frac{1}{T} \, A$, when $T \to 0$. We introduce a
parameter $\beta= \frac{1}{T}$, and we will analyze the Gibbs
probabilities for $\beta \, A$, when $\beta \to \infty.$

From now on, $\mathbf y^*$ is simply denoted by $\mathbf y$ and we identify the spaces $\Sigma_{\mathbf M}$ and $\Sigma_{\mathbf M}^*$.
For each real value $\beta$, we consider the map $ \mathcal G^+_{\beta A} : \mathcal H \to \mathcal H $
and the corresponding Lipschitz function $ \phi^{+}_\beta $, the forward effective potential for $\beta A$,
and the corresponding constant $\lambda^{+}_\beta \in \mathbb{R}$.
For each $\mathbf y$, we consider then the effective probability $\mu_{\mathbf y, \beta A}$ as before.
In order to avoid a heavy notation we will drop the $A$ and the $+$ in this section.

In this way, for each parameter $\beta$, we have the equation
$$ \mathcal G_\beta (\phi_\beta) = \phi_\beta + \lambda_\beta. $$

Recall that, for each $\mathbf y$, we have $ \mathcal G_\beta (\phi_\beta)(\mathbf y) = P_{_{TOP}}( \beta A(\mathbf y,\cdot) + \phi_\beta)$,
where the pressure is consider for the setting in the variable $\mathbf x$.
Therefore, for each $\mathbf y$ and $\beta$, one verifies
$$ \phi_\beta (\mathbf y) + \lambda_{\beta} = P_{_{TOP}}( \beta A(\mathbf y,\cdot) + \phi_\beta). $$

\begin{proposition}
The family $\frac{\phi_\beta}{\beta}$ is equilipchitz.
\end{proposition}

\begin{proof}
For each pair of points $ \mathbf y $ and $ \overline{\mathbf y} $, one has
$$ \left | \phi_\beta (\mathbf y) - \phi_\beta(\overline{\mathbf y}) \right | =
\left | P_{_{TOP}}( \beta A(\mathbf y, \cdot) + \phi_\beta) - P_{_{TOP}}(\beta A(\overline{\mathbf y}, \cdot) + \phi_\beta) \right |. $$
The effective probability $ \mu_{\mathbf y, \beta} $ satisfies
$$ P_{_{TOP}}(\beta A(\mathbf y,\cdot) + \phi_\beta) = \int \left( \beta A(\mathbf y, \cdot) + \phi_\beta \right) \, d\mu_{\mathbf y,\beta} +
h_{\mu_{\mathbf y, \beta }}(\sigma). $$
It is clear that $ P_{_{TOP}}(\beta A(\overline{\mathbf y},\cdot) + \phi_\beta) \geq
\int \left(\beta A(\overline{\mathbf y},\cdot) + \phi_\beta\right)\,d\mu_{\mathbf y,\beta} + h_{\mu_{\mathbf y, \beta }}(\sigma) $.

Therefore, the inequality\footnote{Recall footnote~\ref{multiplicacao M}.}
\begin{multline*}
P_{_{TOP}}(\beta A(\mathbf y,\cdot) + \phi_\beta) - P_{_{TOP}}(\beta A(\overline{\mathbf y},\cdot) + \phi_\beta) \le \\
\le \beta \sup_{\mathbf x \in \Sigma_{\mathbf M}} |A(\mathbf y, \mathbf x) \mathbf M(\mathbf y, \mathbf x) -
A(\overline{\mathbf y}, \mathbf x) \mathbf M(\overline{\mathbf y}, \mathbf x)|
\end{multline*}
yields
$$ P_{_{TOP}}(\beta A(\mathbf y,\cdot) + \phi_\beta) - P_{_{TOP}}(\beta A(\overline{\mathbf y},\cdot) + \phi_\beta) \leq
\beta (\|A\|_0 + \text{Lip}(A)) \, d(\mathbf y, \overline{\mathbf y}). $$

Interchanging the roles of $ \mathbf y $ and $ \overline{\mathbf y} $ in the above reasoning, we get
$$ \left | \phi_\beta(\mathbf y) - \phi_\beta (\overline{\mathbf y}) \right | \leq \beta (\|A\|_0 + \text{Lip}(A)) \, d(\mathbf y, \overline{\mathbf y}), $$
and finally
$$ \text{Lip} \left( \frac{\phi_\beta }{\beta} \right) \leq  \|A\|_0 + \text{Lip}(A). $$
\end{proof}

Remember that the effective potential is unique up to an additive constant.
So we will consider the following condition: we fix a point $ \mathbf y^0 \in \Sigma_{\mathbf M}^*$
and we assume that $ \phi_\beta(\mathbf y^0)=0$ for all $ \beta $.
Via subsequences $ \beta_n \to \infty, $ with $ n \to \infty $, using the previous proposition, we get by the Arzela-Ascoli Theorem that
there exists a continuous function $ V : \Sigma_{\mathbf M}^* \to \mathbb{R}$ such that $ V (\mathbf y^0) = 0 $ and, in the uniform convergence,
$$\frac{\phi_{\beta_n}}{\beta_n} \to V.$$
Since $ \text{Lip}\left(\phi_\beta/\beta\right) \leq \|A\|_0 + \text{Lip}(A) $ implies $ \text{Lip}(V) \leq \|A\|_0 + \text{Lip}(A) $,
the function $ V $ is actually Lipschitz continuous.
Notice that, in principle, such a limit could depend on the chosen subsequence.

\begin{proposition}
Suppose that in the uniform convergence
$$ \frac{\phi_{\beta_n}}{\beta_n} \to V, $$
when $ \beta_n \to \infty $.
Let $\mu_{\mathbf y, \beta_n}$ be the effective probability for the observable $\beta_n A$ at a fixed point $ \mathbf y $.
Then, any accumulation probability measure $ \mu_{\mathbf y}^\infty \in \mathcal M_\sigma $ of the sequence $ \mu_{\mathbf y, \beta_n}$
is a maximizing probability for $ A(\mathbf y,\cdot) + V $, that is,
$$ \int (A(\mathbf y,\cdot) + V) \, d\mu_{\mathbf y}^\infty =
\max_{\mu \in \mathcal M_\sigma} \int (A(\mathbf y,\cdot) + V) \, d\mu. $$
\end{proposition}

\begin{proof}
Take any $\sigma$-invariant probability $\mu$. Thus
\begin{eqnarray*}
\int (\beta_n  A(\mathbf y, \cdot) + \phi_{\beta_n}) \, d \mu + h_{\mu}(\sigma)
& \leq & P_{_{TOP}}(\beta_n A(\mathbf y,\cdot) + \phi_{\beta_n}) \\
& = & \int (\beta_n  A(\mathbf y, \cdot) + \phi_{\beta_n}) \, d \mu_{\mathbf y, \beta_n} + h_{\mu_{\mathbf y, \beta_n }}(\sigma).
\end{eqnarray*}
Given an accumulation probability measure $ \mu_{\mathbf y}^\infty $ of the sequence $ \mu_{\mathbf y, \beta_n} $, from
$$ \int \left(A(\mathbf y, \cdot) + \frac{\phi_{\beta_n}}{\beta_n}\right) \, d \mu + \frac{1}{\beta_n} h_{\mu}(\sigma) \leq
\int \left(A(\mathbf y, \cdot) + \frac{\phi_{\beta_n}}{\beta_n}\right) \, d \mu_{\mathbf y, \beta_n} + \frac{1}{\beta_n} h_{\mu_{\mathbf y, \beta_n}}(\sigma), $$
we get the inequality
$$ \int (A(\mathbf y, \cdot) + V) \, d \mu \leq \int (A(\mathbf y, \cdot) + V) \, d \mu_{\mathbf y}^\infty. $$
Therefore, $ \mu_{\mathbf y}^\infty $ is a maximizing probability for $ A(\mathbf y, \cdot) + V $.
\end{proof}

\begin{proposition}\label{acumulacao}
Assume that in the uniform convergence
$$ \frac{\phi_{\beta_n}}{\beta_n}\to V, $$
when $\beta_n \to \infty$. Suppose also that
$\mu_{\mathbf y, \beta_n}$, the effective probability
for the observable $\beta_n A$ at a fixed point $ \mathbf y $,
converges in the weak* topology to $ \mu_{\mathbf y}^\infty \in \mathcal M_\sigma $.
Then,
\begin{eqnarray*}
\lim_{n \to \infty} \frac{\lambda_{\beta_n}}{\beta_n}
& = & \max_{\mu \in \mathcal M_\sigma}  \int_{\Sigma_{\mathbf M}} \left( A(\mathbf y, \mathbf x) + V(\mathbf x)- V(\mathbf y) \right) \, d\mu(\mathbf x) \\
& = & \int_{\Sigma_{\mathbf M}} \left( A(\mathbf y, \mathbf x) + V(\mathbf x)- V(\mathbf y) \right) \, d\mu_{\mathbf y}^\infty(\mathbf x).
\end{eqnarray*}
\end{proposition}

\begin{proof}
As for any given point $\mathbf y$
$$ \phi_{\beta_n}(\mathbf y) + \lambda_{\beta_n} =
\int \left( \beta_n  A(\mathbf y,\cdot) + \phi_{\beta_n} \right) \, d \mu_{\mathbf y, \beta_n} + h_{\mu_{\mathbf y, \beta_n}}(\sigma), $$
then dividing this expression by $\beta_n$, taking limit, and using last proposition, we immediately get the claim.
\end{proof}

We point out that obviously the limit function $ V \in \mathcal H $ and the limit measure $ \mu_{\mathbf y}^\infty \in \mathcal M_\sigma $ may
depend on the particular choice of the sequence $\beta_n$. Notice the previous proposition guarantees that the value
$ \int_{\Sigma_{\mathbf M}} \left( A(\mathbf y, \cdot) + V - V(\mathbf y) \right) \, d\mu_{\mathbf y}^\infty $ does not
depend on the point $ \mathbf y $. Actually it does not even depend on the function $ V $.

\begin{proposition}
Suppose that in the uniform convergence
$$ \frac{\phi_{\beta_n}}{\beta_n}\to V \quad \text{and} \quad \frac{\phi_{\bar\beta_n}}{\bar\beta_n}\to \overline V, $$
when $\beta_n, \bar\beta_n \to \infty$. Then, for all point $ \mathbf y $,
\begin{multline*}
\max_{\mu \in \mathcal M_\sigma}  \int_{\Sigma_{\mathbf M}} \left( A(\mathbf y, \mathbf x) + V(\mathbf x)- V(\mathbf y) \right) \, d\mu(\mathbf x) = \\
= \max_{\mu \in \mathcal M_\sigma}  \int_{\Sigma_{\mathbf M}} \left( A(\mathbf y, \mathbf x) + \overline V(\mathbf x)- \overline V(\mathbf y) \right) \, d\mu(\mathbf x).
\end{multline*}
\end{proposition}

\begin{proof}
Passing to subsequences if necessary, we use the previous proposition to define
$$ c := \lim_{n \to \infty} \frac{\lambda_{\beta_n}}{\beta_n} \quad \text{and} \quad \bar c := \lim_{n \to \infty} \frac{\lambda_{\bar\beta_n}}{\bar\beta_n}. $$
Notice that, again from proposition~\ref{acumulacao},
\begin{align*}
 V(\mathbf y) + c & = \max_{\mu \in \mathcal M_\sigma}  \int_{\Sigma_{\mathbf M}} \left( A(\mathbf y, \mathbf x) + V(\mathbf x) \right) \, d\mu(\mathbf x) \quad {and} \\
 \overline V(\mathbf y) + \bar c  & = \max_{\mu \in \mathcal M_\sigma}  \int_{\Sigma_{\mathbf M}} \left( A(\mathbf y, \mathbf x) + \overline V(\mathbf x) \right) \, d\mu(\mathbf x),
\end{align*}
for all point $ \mathbf y $. Let $ \mathbf y^0 $ be a global maximum point for $ V - \overline V $. Consider then a probability $ \mu_0 \in \mathcal M_\sigma $ such that
$$ V(\mathbf y^0) + c =  \int_{\Sigma_{\mathbf M}} \left( A(\mathbf y^0, \mathbf x) + V(\mathbf x) \right) \, d\mu_0(\mathbf x). $$
It clearly follows that
\begin{eqnarray*}
V(\mathbf y^0) + c - \overline V(\mathbf y^0) - \bar c & \le &
 \int_{\Sigma_{\mathbf M}} \left[\left( A(\mathbf y^0, \mathbf x) + V(\mathbf x) \right) -  \left( A(\mathbf y^0, \mathbf x) + \overline V(\mathbf x)\right) \right] d\mu_0(\mathbf x) \\
 & = & \int_{\Sigma_{\mathbf M}} \left( V(\mathbf x) - \overline V(\mathbf x) \right) \, d\mu_0(\mathbf x) \le
 V(\mathbf y^0) - \overline V(\mathbf y^0),
\end{eqnarray*}
which shows that $ c \le \bar c $. We can proceed in the same way changing in the reasoning $V$ and $\overline V$. Therefore $ c = \bar c $.
\end{proof}

\begin{theorem}
There exists the limit
$$ c_A := \lim_{\beta \to \infty} \frac{\lambda_{\beta}}{\beta}. $$
\end{theorem}

\begin{proof}
The previous propositions guarantee that $ \{ \lambda_\beta / \beta \}_{\beta > 0} $ has a unique accumulation point as $ \beta $ goes to infinity.
\end{proof}

 In the next section, we explain how the real constant $ c_A $ is related with an ergodic Kantorovich transshipment  problem.

\end{section}

\begin{section}{Ergodic Transshipment}

We remark that one may write, for all limit function $ V \in \mathcal H $ and for any point $ \mathbf y $,
\begin{equation}\label{cAmax}
c_A = \max_{\mu \in \mathcal M_\sigma}  \int_{\Sigma_{\mathbf M}} \left( A(\mathbf y, \mathbf x) + V(\mathbf x)- V(\mathbf y) \right) \, d\mu(\mathbf x).
\end{equation}
Therefore, from ergodic optimization theory, one obtains that
$$ c_A = \inf_{f \in C^0(\Sigma_{\mathbf M})} \sup_{\mathbf x \in \Sigma_{\mathbf M}} \left[A(\mathbf y, \mathbf x) + V(\mathbf x) - V(\mathbf y) + f(\mathbf x) - f(\sigma(\mathbf x))\right]. $$
Moreover, if we fixed a limit function $ V \in \mathcal H $, for each point $ \mathbf y $, there exists a function $ U_{\mathbf y} \in \mathcal H $ (called a \emph{sub-action with respect to $ A(\mathbf y, \cdot) + V - V(\mathbf y)$}) such that
\begin{equation}\label{subacao}
A(\mathbf y, \mathbf x) + V(\mathbf x) - V(\mathbf y) + U_{\mathbf y}(\mathbf x) - U_{\mathbf y}(\sigma(\mathbf x)) \le c_A, \quad \forall \, \mathbf x \in \Sigma_{\mathbf M},
\end{equation}
and the equality holds on the support of the maximizing measure $ \mu_{\mathbf y}^{\infty} $. We refer the reader to
\cite{CLT,CoG,Jenkinson} for details on ergodic optimization theory.

Notice that equation~\eqref{subacao} implies that
$$ V(\mathbf y) + c_A \ge A(\mathbf y, \mathbf x) + V(\mathbf x) + U_{\mathbf y}(\mathbf x) - U_{\mathbf y}(\sigma(\mathbf x)), \quad \forall \, (\mathbf y, \mathbf x) \in \hat \Sigma_{\mathbf M}.  $$
Furthermore, since the equality holds at $ (\mathbf y, \mathbf x) $ whenever $ \mathbf x $ belongs to the support of $ \mu_{\mathbf y}^{\infty} $, one has
\begin{equation}\label{Vautad}
V(\mathbf y) + c_A = \sup_{\mathbf x} \left[ A(\mathbf y, \mathbf x) + V(\mathbf x) + U_{\mathbf y}(\mathbf x) - U_{\mathbf y}(\sigma(\mathbf x)) \right], \quad \forall \, \mathbf y \in \Sigma_{\mathbf M}^*.
\end{equation}
We get from the above equation (see, for instance, \cite{CD}) that  $V$ is an additive eigenfunction and $ c_A $ is an additive eigenvalue for
$$ \mathcal C (\mathbf y, \mathbf x) :=A(\mathbf y, \mathbf x) + U_{\mathbf y} (\mathbf x) - U_{\mathbf y} (\sigma(\mathbf x)), \quad \forall \, (\mathbf y, \mathbf x) \in \hat \Sigma_{\mathbf M}. $$
The question about uniqueness of the $V$ which is solution of an additive problem is not so simple.
For instance, it was considered in section 4 in \cite{LMST}, but it requires some stringent assumptions.

Notice now that, by its very construction, the map $ (\mathbf y, \mathbf x) \mapsto U_{\mathbf y}(\mathbf x) $ may depend on the fixed limit function $ V $. Moreover, we only have information on its Lipschitz regularity on the $ \mathbf x $ variable. In particular, one cannot say \emph{a priori} how the map $ (\mathbf y, \mathbf x) \mapsto  \mathcal C (\mathbf y, \mathbf x) $ varies.

However, it is not difficult to provide examples of observables defining a continuous application $ \mathcal C $ as above.
For instance, considering any $ A_1, A_2 \in \mathcal H $, this is the case for the observable
$$ A(\mathbf y, \mathbf x) = A_1(\mathbf x) + A_2(\mathbf y), \quad \forall \,
 (\mathbf y, \mathbf x) \in \hat \Sigma_{\mathbf M} . $$
Indeed, if $ V \in \mathcal H $ is any possible limit function, let $ U \in \mathcal H $ be a sub-action with respect to $ A_1 + V $,
that is:
$$  A_1(\mathbf x) + V(\mathbf x) + U(\mathbf x) - U(\sigma(\mathbf x)) \le \max_{\mu \in \mathcal M_\sigma}  \int \left( A_1 + V\right) \, d\mu, \quad \forall \, \mathbf x  \in \Sigma_{\mathbf M}. $$
From~\eqref{cAmax}, we then get
$$ A(\mathbf y, \mathbf x)  + V(\mathbf x)  - V(\mathbf y) + U(\mathbf x) - U(\sigma(\mathbf x)) \le c_A $$
everywhere on $  \hat \Sigma_{\mathbf M} $. In particular, we may choose $ U_{\mathbf y} \equiv U $ for all $ \mathbf y $ in such a situation.

In general, by standard selection arguments (see section 2.1 in \cite{Mol} and references therein),  one may always assure the existence of a family of sub-actions $ \{U_{\mathbf y}\}_{\mathbf y}$ for which the corresponding map $ (\mathbf y, \mathbf x) \mapsto  \mathcal C (\mathbf y, \mathbf x) $ is Borel measurable. The main point is to consider just those sub-actions obtained as accumulation functions of eigenfunctions of Ruelle transfer operator when the temperature goes to zero through some fixed sequence (see proposition 29 in \cite{CLT}). Note that these eigenfunctions are continuous on the observable.  We leave the details to the reader. Finally, it is well known in
ergodic optimization theory that these sub-actions have  uniformly bounded  oscillation. Hence, for each fixed limit function $ V $, there exists a family $ \{U_{\mathbf y}\}_{\mathbf y} $ of sub-actions  with respect to $ A(\mathbf y, \cdot) + V - V(\mathbf y)$ such that the the map
$$ (\mathbf y, \mathbf x) \in \hat \Sigma_{\mathbf M} \mapsto  \mathcal C (\mathbf y, \mathbf x) =A(\mathbf y, \mathbf x) + U_{\mathbf y} (\mathbf x) - U_{\mathbf y} (\sigma(\mathbf x)) $$
is Borel measurable and bounded\footnote{Notice that it obviously follows from~\eqref{subacao} that a such map $ \mathcal C $ is bounded from above.}.

We consider from  now on $   \mathcal C $ as a bounded measurable cost function in order to introduce a transshipment problem.
Let then $ \pi : \hat \Sigma_{\mathbf M} \to \Sigma_{\mathbf M} $ and $\pi^* : \hat \Sigma_{\mathbf M} \to \Sigma_{\mathbf M}^* $ be the canonical projections.
We are specially interested  in the set of Borel  probabilities $ \hat{\eta} (d\mathbf y, d\mathbf x)$ on $ \hat{\Sigma}_{\mathbf M} $ verifying $ (\pi)_*  (\hat{\eta} ) = (\pi^*)_*  (\hat{\eta} ) $.

\begin{definition} [{\bf The Ergodic  Kantorovich Transshipment Problem}]
Given $ A : \hat \Sigma_{\mathbf M} \to \mathbb R $ Lipschitz continuous, we are interested in the maximization problem
\begin{eqnarray*}
\kappa_{\text{erg}} & := &  \sup_{ (\pi)_*  (\hat{\eta} ) = (\pi^*)_*  (\hat{\eta} ) } \;  \iint_{\hat \Sigma_{\mathbf M}} \mathcal C(\mathbf y,\mathbf x) \, d\hat{\eta}(\mathbf y,\mathbf x) \\
 & = &  \sup_{ (\pi)_*  (\hat{\eta} ) = (\pi^*)_*  (\hat{\eta} ) } \;  \iint_{\hat \Sigma_{\mathbf M}} [\,A(\mathbf y,\mathbf x) - U_{\mathbf y} (\mathbf x) - U_{\mathbf y} (\sigma(\mathbf x))\,]\, d\hat{\eta}(\mathbf y,\mathbf x).
\end{eqnarray*}
An ergodic transshipment measure for $A$ is a probability $\hat\eta$  on $ \hat{\Sigma}_{\mathbf M} $, with $  (\pi)_*  (\hat{\eta} ) = (\pi^*)_*  (\hat{\eta} ) $, that attains such a supremum.
\end{definition}

We point out that the classical transport or transshipment problems do not have an intrinsic ergodic nature.
Note that $ \mathcal C $ has a dynamical character.
We refer the reader to \cite{Ra} for general results (not of ergodic nature) on transshipment.
In \cite{LOT}, it is consider an ergodic transport problem.

We claim that $ c_A = \kappa_{\text{erg}} $, or in a more self-contained statement:

\begin{theorem} For the Lipschitz observable $ \beta A $, $ \beta>0$, consider its forward effective potential  $ \phi_\beta^+ $
(normalized by $ \phi_\beta^+(\mathbf y^0) = 0 $) and its  effective constant $ \lambda_\beta^+ $. Assume that in the uniform convergence
$ \phi_{\beta_n}^+ / \beta_n \to V $,
when $\beta_n \to \infty$.  Then there exists a family $ \{U_{\mathbf y}\}_{\mathbf y} $ of sub-actions with respect to
 $ A(\mathbf y, \cdot) + V - V(\mathbf y)$ such that
$$ \lim_{\beta \to \infty} \frac{\lambda_\beta^+}{\beta} =  \sup_{ (\pi)_*  (\hat{\eta} ) = (\pi^*)_*  (\hat{\eta} ) } \;   \iint_{\hat \Sigma_{\mathbf M}} [\,A(\mathbf y,\mathbf x) - U_{\mathbf y} (\mathbf x) - U_{\mathbf y} (\sigma(\mathbf x))\,]\, d\hat{\eta}(\mathbf y,\mathbf x). $$
\end{theorem}

\begin{proof}
We remark that inequality~\eqref{subacao} implies that $ \kappa_{\text{erg}} \le c_A $. Indeed,
given any Borel probability  $ \hat{\eta}$ on $ \hat{\Sigma}_{\mathbf M} $ such that $ (\pi)_*  (\hat{\eta} ) = (\pi^*)_*  (\hat{\eta} ) $, one clearly has
\begin{multline*}
\iint_{\hat \Sigma_{\mathbf M}} [\,A(\mathbf y,\mathbf x) - U_{\mathbf y} (\mathbf x) - U_{\mathbf y} (\sigma(\mathbf x))\,]\, d\hat{\eta}(\mathbf y,\mathbf x) = \\
=\iint_{\hat \Sigma_{\mathbf M}} [\,A(\mathbf y,\mathbf x) - U_{\mathbf y} (\mathbf x) - U_{\mathbf y} (\sigma(\mathbf x)) + V(\mathbf x) - V(\mathbf y)\,]\, d\hat{\eta}(\mathbf y,\mathbf x) \le c_A.
\end{multline*}

Recall that functional equation~\eqref{Vautad} shows the limit $V$ is an additive eigenfunction and the constant $c_A$ is an additive eigenvalue for $\mathcal C$. Actually, since $ \mathcal C $ is bounded,  it is easy to obtain that $ c_A $ is uniquely determined by
$$ c_A = \sup_{\{\mathbf z^k\}_{k \ge 1}} \, \limsup_{k \to \infty} \, \frac{\mathcal C(\mathbf z^1, \mathbf z^2) + \mathcal C(\mathbf z^2, \mathbf z^3) + \cdots + \mathcal C(\mathbf z^k, \mathbf z^{1})}{k}, $$
where the supremum is taken among sequences $ \{\mathbf z^k\}  $ of points of $ \Sigma_{\mathbf M}  \simeq  \Sigma_{\mathbf M}^* $.  See theorem 2.1 in \cite{Ba} for a general result. Notice now that
$$ \frac{\mathcal C(\mathbf z^1, \mathbf z^2) + \mathcal C(\mathbf z^2, \mathbf z^3) + \cdots + \mathcal C(\mathbf z^k, \mathbf z^{1})}{k} =  \iint_{\hat \Sigma_{\mathbf M}} \mathcal C(\mathbf y,\mathbf x) \, d\hat{\eta}_k(\mathbf y,\mathbf x), $$
where $ \hat{\eta}_k $ is the Borel probability on $ \hat \Sigma_{\mathbf M} $ defined by
$$  \hat{\eta}_k := \frac{1}{k} \delta_{(\mathbf z^1, \mathbf z^2)} + \frac{1}{k} \delta_{(\mathbf z^2, \mathbf z^3)} + \ldots + \frac{1}{k} \delta_{(\mathbf z^k, \mathbf z^{1})}. $$
Since  $  (\pi)_*  (\hat{\eta}_k ) = (\pi^*)_*  (\hat{\eta}_k ) $ for all $ k \ge 1 $, it obviously follows that $ c_A \le \kappa_{erg} $.
\end{proof}

\end{section}

\begin{section}{Contraction properties of $ \mathcal G^+ $}

We would like to discuss now the proof of Theorem~\ref{principal}.
We start pointing out an immediate contraction property of $ \mathcal G^+ $.

\begin{proposition}
For all $ \phi, \psi \in \mathcal H $,
$$ \left \| \mathcal G^+(\phi) - \mathcal G^+(\psi) \right \|_0 \le \left \| \phi - \psi \right \|_0. $$
\end{proposition}

\begin{proof}
Given $ \mathbf y \in \Sigma_{\mathbf M}^* $, take $ \mu_{\mathbf y} \in \mathcal M_\sigma $ satisfying
$$ \mathcal G^+(\phi)(\mathbf y) = \int_{\Sigma_{\mathbf M}} \left( A(\mathbf y, \mathbf x) + \phi(\mathbf x)\right) \; d\mu_{\mathbf y}(\mathbf x) +
h_{\mu_{\mathbf y}}(\sigma). $$
Obviously
$ \mathcal G^+(\psi)(\mathbf y) \ge \int_{\Sigma_{\mathbf M}} \left( A(\mathbf y, \mathbf x) + \psi(\mathbf x)\right) \; d\mu_{\mathbf y}(\mathbf x) +
h_{\mu_{\mathbf y}}(\sigma) $.
Therefore, we have
$$ \mathcal G^+(\phi)(\mathbf y) - \mathcal G^+(\psi)(\mathbf y) \le
\int_{\Sigma_{\mathbf M}} \left( \phi(\mathbf x) - \psi(\mathbf x) \right) \; d\mu_{\mathbf y}(\mathbf x) \le
\| \phi - \psi \|_0. $$
Since $ \phi $ and $ \psi $ play symmetrical roles, the proof is complete.
\end{proof}

Notice that, for any real number $ \gamma $, we have $ \mathcal G^+(\phi + \gamma) = \mathcal G^+(\phi) + \gamma $.
Let us now identify all functions belonging $ \mathcal H $ which are equal up to an additive constant.
So if we introduce the norm
$$ \| \phi \|_c := \inf_{\gamma \in \mathbb R} \| \phi + \gamma \|_0 $$
for each equivalence class $ \phi \in \mathcal H/\text{constants} $, a fine contraction property can be verified.

\begin{theorem}\label{Gomes-Valdinoci contraction}
Consider $ \phi, \psi \in \mathcal H $ satisfying $ \text{Lip}(\phi), \text{Lip}(\psi) \le K $ for some fixed constant $ K > 0 $.
Then, there exist constants $ C = C(K) > 0 $ and $ \alpha = \alpha(K) > 0 $ such that
$$ \| \mathcal G^+(\phi) - \mathcal G^+(\psi) \|_c \le
\left(1 - C\| \phi - \psi \|_c^\alpha \right) \| \phi - \psi \|_c $$
\end{theorem}

We will need the following lemma.

\begin{lemma}\label{regularidade dimensao}
Let $ A : \hat \Sigma_{\mathbf M} \to \mathbb R $ be Lipschitz continuous observable. Suppose $ \phi \in \mathcal H $
satisfies $ \text{Lip}(\phi) \le K $ for a constant $ K > 0 $.
Given a point $ \mathbf y \in \Sigma^*_{\mathbf M} $, let $ \mu_{\mathbf y} \in \mathcal M_\sigma $ be
the equilibrium state associated to $ A(\mathbf y, \cdot) + \phi \in \mathcal H $. Then there exist constants
$ \Gamma = \Gamma(K) > 0 $ and $ \alpha = \alpha(K) > 0 $ such that, if $ B_\rho \subset \Sigma_{\mathbf M} $
denotes an arbitrary ball of radius $ \rho > 0 $,
$$ \mu_{\mathbf y}(B_\rho) \ge \Gamma \rho^{\alpha}. $$
\end{lemma}

\begin{proof}
Let $ \mu_{\Psi} \in \mathcal M_\sigma $ be the equilibrium measure associated to $ \Psi \in \mathcal H $.
It is well known that $ \mu_{\Psi} $ is a Gibbs state. As a matter of fact, if $ \mathbf x $ is a point belonging to a
ball $ B_{\Lambda^n} $ of radius $ {\Lambda^n} $, from the very proof of the Gibbs property one can obtain
\begin{multline*}
\exp\left[ - \text{Lip}(\Psi) R(\Lambda) - I_{\mathbf M} ( \text{Lip}(\Psi) + h_{_{TOP}}(\sigma) ) S(\Lambda) \right] \le \\
\le \frac{\mu_{\Psi}(B_{\Lambda^n})}{{\exp\left[\sum_{j=0}^{n - 1} (\Psi - P_{_{TOP}}(\Psi))\circ\sigma^j(\mathbf x)\right]}},
\end{multline*}
where $ R $ and $ S $ are rational functions with $ R(0,1), S(0,1) \subset (0, +\infty) $, $ I_{\mathbf M} $ is a positive integer
depending only on the irreducible transition matrix $ \mathbf M $ and  $ h_{_{TOP}}(\sigma) $ denotes the topological entropy.
For details we refer the reader to \cite{PP}.

From the variational principle, one has $  \Psi - P_{_{TOP}}(\Psi) \ge - \text{Lip}{\Psi} - h_{_{TOP}}(\sigma) $.
Therefore, we immediately get
$$ \exp\left[ - \text{Lip}(\Psi) R(\Lambda) - ( \text{Lip}(\Psi) + h_{_{TOP}}(\sigma) ) (I_{\mathbf M}  S(\Lambda) + n) \right]
\le \mu_{\Psi}(B_{\Lambda^n}). $$

Thus, applying this inequality to $ \Psi = A(\mathbf y, \cdot) + \phi $, it is straightforward to verify that
$$ \Gamma(K) \Lambda^{n \alpha(K)} \le \mu_{\mathbf y}(B_{\Lambda^n}), $$
where
$$ \alpha(K):= \frac{Lip(A) + K + h_{_{TOP}}(\sigma)}{\log \Lambda^{-1}} \quad \text{and} $$
$$ \Gamma(K):= \exp\left[ -(\text{Lip}(A) + K)R(\Lambda) - I_{\mathbf M}(\text{Lip}(A) + K + h_{_{TOP}}(\sigma)) S(\Lambda) \right]. $$
\end{proof}

\begin{proof}[Proof of Theorem~\ref{Gomes-Valdinoci contraction}]
Obviously, for  $ \phi \in \mathcal H $ and $ \gamma \in \mathbb R $, we have $ \| \phi + \gamma \|_c = \| \phi \|_c. $
Moreover, given $ \phi, \psi \in \mathcal H $, there exists $ \overline \gamma \in \mathbb R $ such that
$ \| \phi - \psi \|_c = \| \phi - \psi + \overline \gamma \|_0 $.

As $ \mathcal G $ commutes with constants, replacing $ \psi $ by $ \psi - \min \psi $ and
$ \phi $ by $ \phi + \overline \gamma - \min \psi $, without loss of generality, we may assume
$$ \min \psi = 0 \;\; \text{ and } \;\; \| \phi - \psi \|_c = \| \phi - \psi \|_0. $$
We suppose yet $ \phi \neq \psi $, since otherwise there is nothing to argue.

Take then $ \mathbf y \in \Sigma_{\mathbf M}^* $ satisfying
$$ \| \mathcal G^+(\phi) - \mathcal G^+(\psi) \|_0 = | \mathcal G^+(\phi)(\mathbf y) - \mathcal G^+(\psi)(\mathbf y) |. $$
By interchanging the roles of $ \phi $ and $ \psi $ if necessary, we suppose that
$$ | \mathcal G^+(\phi)(\mathbf y) - \mathcal G^+(\psi)(\mathbf y) | = \mathcal G^+(\phi)(\mathbf y) - \mathcal G^+(\psi)(\mathbf y). $$

Since $ \min \psi = 0 $, taking any point $ \mathbf x \in \Sigma_{\mathbf M} $, we get
$$ \|\phi - \psi\|_c \le \|\phi - \phi(\mathbf x) - \psi\|_0 \le \|\phi - \phi(\mathbf x)\|_0 + \|\psi\|_0 \le \text{Lip}(\phi) + \text{Lip}(\psi) \le 2K. $$

Note that $ \| \phi - \psi \|_c = \| \phi - \psi \|_0 $ implies $ \min(\phi - \psi) = - \max(\phi - \psi) $. In particular,
$ \min(\phi - \psi) = - \| \phi - \psi \|_c $. So there exists a point $ \overline{\mathbf x} \in \Sigma_{\mathbf M} $ such that
$$ (\phi - \psi) (\overline{\mathbf x}) = - \| \phi - \psi \|_c < 0. $$

Hence, when $ \mathbf x \in \Sigma_{\mathbf M} $ verifies $ d(\mathbf x, \overline{\mathbf x}) \le \frac{\|\phi - \psi\|_c}{4K} $, we obtain
\begin{eqnarray}
\phi(\mathbf x) - \psi(\mathbf x)
& \le & |\phi(\mathbf x) - \phi(\overline{\mathbf x})| + |\psi(\overline{\mathbf x}) - \psi(\mathbf x)| + (\phi - \psi)(\overline{\mathbf x}) \nonumber \\
& \le & 2K  \frac{\|\phi - \psi\|_c}{4K} - \| \phi - \psi \|_c \nonumber \\
& = & - \frac{\| \phi - \psi \|_c}{2}< 0. \label{valores negativos}
\end{eqnarray}

Let then $ \mu_{\mathbf y} \in \mathcal M_\sigma $ be such that
$$ \mathcal G^+(\phi)(\mathbf y) = \int_{\Sigma_{\mathbf M}} \left( A(\mathbf y, \mathbf x) + \phi(\mathbf x)\right) \; d\mu_{\mathbf y}(\mathbf x) +
h_{\mu_{\mathbf y}}(\sigma). $$
As in the previous proposition, it follows
$$ \mathcal G^+(\phi)(\mathbf y) - \mathcal G^+(\psi)(\mathbf y) \le
\int_{\Sigma_{\mathbf M}} \left( \phi(\mathbf x) - \psi(\mathbf x) \right) \; d\mu_{\mathbf y}(\mathbf x). $$

So if $ B_{\frac{\|\phi - \psi\|_c}{4K}}(\overline{\mathbf x}) $ denotes the closed ball of radius $ \frac{\|\phi - \psi\|_c}{4K} \in (0, 1) $ and center
$ \overline{\mathbf x} \in \Sigma_{\mathbf M} $, from~\eqref{valores negativos} and lemma~\ref{regularidade dimensao}, we verify
\begin{eqnarray*}
\mathcal G^+(\phi)(\mathbf y) - \mathcal G^+(\psi)(\mathbf y) & \le &
\int_{\Sigma_{\mathbf M} - B_{\frac{\|\phi - \psi\|_c}{4K}}(\overline{\mathbf x})} \left(\phi(\mathbf x) - \psi(\mathbf x)\right)\;d\mu_{\mathbf y}(\mathbf x) \\
& \le & \|\phi - \psi\|_0 \left( 1 - \mu_{\mathbf y}\left(B_{\frac{\|\phi - \psi\|_c}{4K}}(\overline{\mathbf x})\right)\right) \\
& \le & \|\phi - \psi\|_0 \left( 1 - C\|\phi - \psi\|_c^\alpha\right),
\end{eqnarray*}
where $ C := \Gamma/(4K)^\alpha > 0 $.

As $ \| \mathcal G^+(\phi) - \mathcal G^+(\psi) \|_c \le \| \mathcal G^+(\phi) - \mathcal G^+(\psi) \|_0 =
\mathcal G^+(\phi)(\mathbf y) - \mathcal G^+(\psi)(\mathbf y) $, the proof is complete.
\end{proof}

Theorem~\ref{principal} is then a direct consequence of Theorem~\ref{Gomes-Valdinoci contraction}, the fact that
$$ \text{Lip}(\mathcal G^+ (\phi)) \le \| A \|_0 + \text{Lip}(A) \quad \forall \; \phi \in \mathcal H, $$
and the following fixed point theorem due to D. A. Gomes and E. Valdinoci (for a proof, see Appendix A of \cite{GV}).

\begin{theoremA}
Let $ \mathbf F $ be a closed subset of a Banach space, endowed with a norm $ \| \cdot \| $.
Suppose that $ G : \mathbf F \to \mathbf F $ is so that
$$ \| G(\phi) - G(\psi) \| \le \left(1 - C\| \phi - \psi \|^\alpha \right) \| \phi - \psi \|, $$
for all $ \phi, \psi \in \mathbf F $ and some given constants $ C, \alpha > 0 $. Then there exists a unique
$ \phi^+ \in \mathbf F $ such that $ G(\phi^+) = \phi^+ $. Moreover, given any $ \phi_0 \in \mathbf F $,
we have
$$ \phi^+ = \lim_{n \to +\infty} G^n(\phi_0). $$
\end{theoremA}

\end{section}

\footnotesize{

}

\end{document}